\documentclass[letterpaper,11pt]{amsart}

\usepackage[margin=1.2in]{geometry}
\usepackage{amsmath,amsthm,amssymb}
\usepackage{xspace,xcolor}
\usepackage[breaklinks,colorlinks,citecolor=teal,linkcolor=teal,urlcolor=teal,pagebackref,hyperindex]{hyperref}
\usepackage[alphabetic]{amsrefs}
\usepackage[all]{xy}
\usepackage{color}
\usepackage{tikz-cd}

\theoremstyle{plain}
\newtheorem{thm}{Theorem}[section]

\newtheorem{prop}[thm]{Proposition}

\theoremstyle{definition}
\newtheorem{defi}[thm]{Definition}

\newtheorem{eg}[thm]{Example}

\newtheorem{question}[thm]{Question}

\theoremstyle{remark}
\newtheorem{rmk}[thm]{Remark}

\def\cD{\mathcal{D}}

\def\cI{\mathcal{I}}

\def\cO{\mathcal{O}}

\def\fa{\mathfrak{a}}
\def\fb{\mathfrak{b}}

\def\fm{\mathfrak{m}}

\def\frm{\mathfrak{m}}

\def\.{\cdot}
\def\^{\widehat}

\def\({\left(}
\def\){\right)}

\renewcommand{\and}{ \ \ \text{ and } \ \ }

\newcommand{\factor}[2]{\left. \raise 2pt\hbox{$#1$} \right/\hskip -2pt\raise -2pt\hbox{$#2$}}

\DeclareMathOperator{\lct} {lct}

\makeatletter
\@namedef{subjclassname@2020}{\textup{2020} Mathematics Subject Classification}
\makeatother

\usepackage{soul}

\begin{document}

\author[M.~Musta\c{t}\u{a}]{Mircea Musta\c{t}\u{a}}

\address{Department of Mathematics, University of Michigan, 530 Church Street, Ann Arbor, MI 48109, USA}

\email{mmustata@umich.edu}

\thanks{The author was partially supported by NSF grant DMS-2301463 and by the Simons Collaboration grant \emph{Moduli of
Varieties}}

\subjclass[2020]{14B05, 14J17, 32S25}

\begin{abstract}
Given a hypersurface defined by $f$ in a smooth complex algebraic variety $X$, and a point $P$ on this hypersurface, we consider the invariant $\beta_P(f)$
given by the log canonical threshold at $P$ of $\fm_P\cdot J_f$, where $\fm_P$ is the ideal defining $P$ and $J_f$ is the Jacobian ideal of $f$.
We show that this invariant satisfies most of the formal properties of the log canonical threshold of $f$ and give some examples. Dano Kim \cite{DanoKim} asked whether
this invariant always gives an upper bound for the minimal exponent of $f$ at $P$. Motivated by this, we raise another question about minimal exponents, give a
positive answer to a weaker version, and discuss some examples. 
\end{abstract}

\title[On some invariants of hypersurface singularities]{On some invariants of hypersurface singularities}

\dedicatory{To Bernard Teissier, with admiration, on the occasion of his 80th birthday}

\baselineskip 16pt \footskip = 32pt

\maketitle

\section{Introduction}
Let $X$ be a smooth, irreducible, $n$-dimensional complex algebraic variety, and let $P\in X$. Suppose that 
$f\in\cO_X(X)$ is nonzero, defining the hypersurface $Z$, with $P\in Z$. An important invariant which measures the singularity of $Z$ at $P$ is the 
log canonical threshold ${\rm lct}_P(f)$. This plays an important role in birational geometry due to the fact that it is the largest $t>0$ such that the pair $(X,tZ)$
is log canonical. The log canonical threshold at $P$ can be defined, more generally, for nonzero ideals that vanish at $P$. 

A refinement of the log canonical threshold is provided by the minimal exponent $\widetilde{\alpha}_P(f)$, introduced in this generality by Saito 
\cite{Saito-B}, but which was extensively studied in the 80s by Varchenko, Steenbrink, Malgrange, Teissier, Loeser, etc. for hypersurfaces with isolated singularities,
as the \emph{Arnold exponent}. The minimal exponent $\widetilde{\alpha}_P(f)$ determines the log canonical threshold via the formula
$${\rm lct}_P(f)=\min\big\{\widetilde{\alpha}_P(f),1\big\}.$$
Therefore, the minimal exponent is interesting when ${\rm lct}_P(f)=1$, that is, when the pair $(X,Z)$ is log canonical. 
We note that $\widetilde{\alpha}_P(Z)=\infty$ if and only if $P$ is a smooth point of $Z$, and Saito showed in \emph{loc. cit.} that $\widetilde{\alpha}_P(f)>1$
if and only if $Z$ has rational singularities at $P$. 

In \cite{Teissier}, Bernard Teissier conjectured the following lower bound when $Z$ has an isolated singularity at $P$: 
$$\widetilde{\alpha}_P(f)\geq \tfrac{1}{\theta_P(f)+1}+\tfrac{1}{\theta_P(f\vert_{H_1})+1}+\ldots+\tfrac{1}{\theta_P(f\vert_{H_1\cap\ldots\cap H_{n-1}})},$$
where $H_1,\ldots,H_{n-1}$ are general smooth hypersurfaces in $X$ containing $P$. Here the $\theta$-invariant is defined as follows on $X$
(and similarly for each restriction): $\theta_P(f)$ is the minimum of the positive rational numbers $\tfrac{r}{s}$ with the property that 
${\mathfrak m}_P^r$ is contained in the integral closure $\overline{J_f^s}$ of $J_f^s$ in a neighborhood of $P$, where ${\mathfrak m}_P$ is the ideal defining $P$
and $J_f$ is the Jacobian ideal of $f$. This note is motivated by a result of Dano Kim \cite{DanoKim}, who proved the bound in Teissier's conjecture, not for
the minimal exponent, but for ${\rm lct}_P({\mathfrak m}_PJ_f)$. In the meantime, Teissier's conjecture was settled in \cite{DM}, by refining Loeser's approach in \cite{Loeser}.
However, this still leaves open the question concerning the comparison of the two inequalities: 

\begin{question}\label{question_Kim}(\cite{DanoKim})
With the above notation, if $P\in Z$ is a singular point, do we always have 
\begin{equation}\label{eq_Dano_Kim}
\widetilde{\alpha}_P(f)\leq {\rm lct}_P({\mathfrak m}_PJ_f)?
\end{equation}
\end{question}

We mention that there is a lower bound for the minimal exponent in terms of the Jacobian ideal, given by \cite[Theorem~1.2]{CKM}: we have $\widetilde{\alpha}_P(f)\geq {\rm lct}_P(f,J_f^2)$.

In Question~\ref{question_Kim}, we do not need to assume that $Z$ has an isolated singularity at $P$, but it is easy to see that it is enough to give a positive answer in this case
(see Remark~\ref{isolated_singularity} below). It was noticed in \cite{DanoKim} that since $f$, around $P$, lies in the integral closure of ${\mathfrak m}_PJ_f$,
we always have 
$${\rm lct}_P(f)\leq {\rm lct}_P({\mathfrak m}_PJ_f),$$
hence the above question clearly has a positive answer when $Z$ doesn't have a rational singularity at $P$. 
We make a few remarks and give a few examples when we have a positive answer to Question~\ref{question_Kim} in Section~\ref{section3}.

A positive answer to this question would be interesting since it would provide an upper bound for the minimal exponent in terms of 
the conceptually easier to handle right-hand side. While lower bounds for the minimal exponents are known using log resolutions (see \cite[Corollary~D]{MPV}),
it is not easy to provide upper bounds. The only available upper bound is a basic one involving the multiplicity at $P$: if ${\rm mult}_P(f)=d\geq 2$, then
$\widetilde{\alpha}_P(f)\leq\tfrac{n}{d}$ (see \cite[Theorem~E(3)]{MPV}), as well as a weighted version of this one (see \cite[Proposition~2.1]{CDM}).

Motivated by Question~\ref{question_Kim}, in Section~\ref{section2} we consider in more detail the invariant ${\rm lct}_P({\mathfrak m}_PJ_f)$ that, to simplify notation, we denote by $\beta_P(f)$.
As we will see, the general properties of log canonical thresholds of ideals imply that $\beta_P(f)$ behaves similarly to the log canonical threshold (or the minimal exponent). 

In a different direction, but still motivated by Question~\ref{question_Kim}, we pose the following

\begin{question}\label{question2}
If $f\in \fm_P\fa$, for some ideal $\fa\subseteq \fm_P$, does it follow that 
$$\widetilde{\alpha}_P(f)\leq {\rm lct}_P(\fm_P\fa)?$$
\end{question}

For a more optimistic question, whose positive answer would imply the assertion in Question~\ref{question_Kim}, see Question~\ref{question3} below.
We show in Theorem~\ref{prop_question2} that under the hypothesis in Question~\ref{question2}, we have the weaker
inequality $\widetilde{\alpha}_P(f)\leq {\rm lct}_P(\fa)$. We also discuss the case when $\fa$ is a monomial ideal in Example~\ref{example_monomial}. 

\subsection*{Acknowledgment}
It is a pleasure to dedicate this note to Bernard Teissier, whose work has had a profound influence on Singularity Theory and whose kind and encouraging personality has inspired generations of mathematicians to enter this field.

\section{The invariant $\beta_P(f)$}\label{section2}

Let $X$ be a smooth, irreducible, complex algebraic variety of dimension $n$. Suppose that $f\in\cO_X(X)$ and $P\in X$ are such that $f(P)=0$.
Let $Z$ be the hypersurface defined by $f$. We denote by $\fm_P$ the ideal defining $P$ in $X$.
For the definition and basic facts about log canonical thresholds of ideals in $X$, we refer to \cite[Chapter~9]{Lazarsfeld} or \cite{Mustata2}.

Recall that the Jacobian ideal $J_f$ is defined as follows. For every point $Q\in X$, we choose algebraic coordinates\footnote{This means that $dx_1,\ldots,dx_n$ trivialize
the cotangent sheaf.} $x_1,\ldots,x_n$ in some neighborhood $U$ of $Q$, and then $J_f$ is defined in $U$ by the ideal $\big(f,\tfrac{\partial f}{\partial x_1},\ldots,\tfrac{\partial f}{\partial x_n}\big)$. 
It is easy to see that this ideal only depends on $Z$, and not on the
choice of coordinates or on the function $f$ defining $Z$.

\begin{defi}
With the above notation, we put
$$\beta_P(Z)=\beta_P(f):={\rm lct}_P(\fm_PJ_f).$$
\end{defi}

\begin{rmk}
If $Z$ is any hypersurface in $X$, then we can still define $\beta_P(Z)$ for every $P\in Z$, by choosing a defining equation around $P$, and using the fact that
$J_f$ is independent of the choice of $f$.
\end{rmk}

\begin{rmk}\label{version_Jacobian_ideal}
Another version of Jacobian ideal is the ideal $J'_f$, defined around $P\in Z$ by $\big(\tfrac{\partial f}{\partial x_1},\ldots,\tfrac{\partial f}{\partial x_n}\big)$,
where $x_1,\ldots,x_n$ are algebraic coordinates in a neighborhood of $P$. This is again independent of the choice of coordinates, but it depends on the
choice of equation $f$ for $Z$. However, $J'_f$ and $J_f$ have the same integral closure in a neighborhood of $P$, hence the definition of $\beta_P(Z)$ does not
change if we use $J'_f$ instead of $J_f$ by \cite[Corollary~9.6.17]{Lazarsfeld}.
\end{rmk}

\begin{eg}
We have $\beta_P(Z)\leq n$, with equality if and only if $P$ is a smooth point of $Z$. This is a consequence of the well-known fact that if $P$ lies in the zero-locus of an ideal $\fa$,
then $\lct_P(\fa)\leq n$, with equality if and only if $\fa=\fm_P$ in a neighborhood of $P$. 
\end{eg}

\begin{rmk}\label{rmk_bound_by_integral_closure}
It is well-known that in a neighborhood of $P$, $f$ lies in the integral closure of $\fm_PJ_f$, hence \emph{loc. cit.} gives
$${\rm lct}_P(f)\leq\beta_P(f).$$
\end{rmk}

\begin{eg}
If $f=x_1^d\in {\mathbf C}[x_1,\ldots,x_n]$, for some $d\geq 2$, then $\fm_0J_f=(x_1,\ldots,x_n)x_1^{d-1}$. We see that a log resolution is given
by the blow-up at the origin and we have
$$\beta_0(f)=\min\big\{\tfrac{1}{d-1},\tfrac{n}{d}\big\}=\left\{
\begin{array}{cl}
\tfrac{1}{d}, & {\rm if}\,n=1; \\[2mm]
\tfrac{1}{d-1}, & {\rm if}\,n\geq 2.
\end{array}\right.$$
We note that the result depends on the ambient ring we consider. This behavior is one aspect in which $\beta_P(f)$ differs from ${\rm lct}_P(f)$,
but it is not surprising, due to the presence of the maximal ideal in the definition of $\beta_P$
\end{eg}

\begin{eg}
Suppose that $a$ and $b$ are positive integers, with $a\geq b\geq 1$, and let $f=x^ay^b\in {\mathbf C}[x,y]$.
In this case, we have $\fm_0J_f=x^{a-1}y^{b-1}(x,y)^2$, hence a log resolution of $\fm_0J_f$ is given by the blow-up at the origin, and we see
that 
$$\beta_0(f)=\min\big\{\tfrac{1}{a-1},\tfrac{2}{a+b}\big\},$$
hence $\beta_0(f)=\tfrac{1}{a-1}$ if $a\geq b+2$, and $\beta_0(f)=\tfrac{2}{a+b}$ if $a=b$ or $a=b+1$.
\end{eg}

\begin{eg}\label{eg_ordinary}
Suppose that $Z$ has an ordinary singular point at $P$, that is, the projective tangent cone of $Z$ at $P$ is smooth. If $\pi\colon Y\to X$ is the blow-up of $X$ at $P$,
with exceptional divisor $E$, then $J_f\cdot\cO_Y=\cO_Y\big(-(d-1)E\big)$ around $P$, where $d={\rm mult}_P(Z)$, and $\fm_P\cdot\cO_Y=\cO_Y(-E)$, hence $\pi$ is a log resolution of
$\fm_PJ_f$ in a neighborhood of $P$, and $\beta_P(Z)=\tfrac{n}{d}$. 
\end{eg}

\begin{eg}\label{det}
Let $A=(x_{ij})_{1\leq i,j\leq r}$ and let $f={\rm det}(A)\in {\mathbf C}[x_{ij}\mid 1\leq i,j\leq r]$. In this case, we have $J_f=I_{r-1}(A)$, the ideal generated
by the $(r-1)$-minors of $A$. It is a result that goes back to \cite{det} that we get a log resolution of $({\mathbf A}^{r^2}, J_f)$ by successively blowing-up
the (strict transforms of the) loci of matrices of rank $\leq i$, for $i=0,\ldots,r-2$. It is clear that this also gives a log resolution of $\fm_0J_f$, and 
an easy computation gives 
$\beta_0(f)=4$.
\end{eg}

In the case of an isolated singularity, we have the following lower bound in terms of the Milnor number. Recall that if $Z$ has an isolated singularity at $P$
(that is, in some neighborhood of $P$, the only possible singular point of $Z$ is $P$), then the Milnor number of $Z$ at $P$ is given by
$\mu_P(Z)=\ell(\cO_{X,P}/J'_f)$, where $J'_f$ is the version of the Jacobian ideal discussed in Remark~\ref{version_Jacobian_ideal}.

\begin{prop}\label{prop_Milnor_number}
If $Z$ is a hypersurface in a smooth, $n$-dimensional variety $X$, which has an isolated singularity at $P$,
then
$$\beta_P(Z)\geq\tfrac{n}{1+\mu_P(Z)^{1/n}}.$$
\end{prop}

\begin{proof}
In what follows, we denote by $e(\fa)$ the Hilbert-Samuel multiplicity of a zero-dimensional ideal. 
Let $f$ be a local equation of $Z$ around $P$. By Remark~\ref{version_Jacobian_ideal}, we have
$$\beta_P(Z)={\rm lct}_P(\fm_PJ'_f),$$ hence by
\cite[Theorem~0.1]{dFEM}, we have
$$\beta_P(Z)\geq \tfrac{n}{e(\fm_PJ'_f)^{1/n}}.$$
On the other hand, it follows from Teissier's Minkowski-type inequality for multiplicities \cite{Teissier0} that
$$e(\fm_PJ'_f)^{1/n}\leq e(\fm)^{1/n}+e(J'_f)^{1/n}=1+e(J'_f)^{1/n}.$$
Finally, since $J'_f$ is generated by a regular sequence, it follows from \cite[Theorem~14.11]{Matsumura}
that $e(J'_f)=\mu_P(Z)$. By putting all these together, we get the inequality in the proposition.
\end{proof}

\begin{eg}
If $Z$ is the affine cone over a smooth projective hypersurface of degree $d$ in ${\mathbf P}^{n-1}$, then 
$\mu_0(Z)=(d-1)^n$, and it follows from Example~\ref{eg_ordinary} that in this case the inequality in
Proposition~\ref{prop_Milnor_number} is an equality.
\end{eg}

We next show that the invariant $\beta_P(Z)$ satisfies some of the general properties of ${\rm lct}_P(Z)$. As we will see,
the basic properties can be easily deduced from the general properties of log canonical thresholds of ideals.

\begin{prop}\label{semicont0}
If $Z$ is a hypersurface in the smooth variety $X$, then the set 
$$\big\{\beta_P(Z)\mid P\in Z\big\}$$
is finite, and for every $c\geq 0$, the set
$$\big\{P\in Z\mid\beta_P(Z)\geq c\big\}$$
is open in $Z$.
\end{prop}

\begin{proof}
Let $\pi\colon X\times Z\to Z$ be the projection and let $s\colon Z\to X\times Z$ be the section of $\pi$ given by
$s(x)=(x,x)$. We consider $\fa=\cI_{\Delta}\cdot \cO_{X\times Z}$, where $\cI_{\Delta}\subseteq\cO_{X\times X}$
is the ideal defining the diagonal. It follows that if $\fb=(J_f\cdot\cO_{X\times Z})\cdot \fa$, then 
$$\beta_P(Z)={\rm lct}_{s(P)}(\fb\cdot\cO_{\pi^{-1}(P)}).$$
The assertion in the proposition then follows by well-known properties of the log canonical threshold in families
(see, for example, \cite[Lemma~4.8 and Theorem~4.9]{Mustata}).
\end{proof}

\begin{rmk}
In light of the above proposition, it makes sense to define a global version of the invariant $\beta_P(Z)$ by putting
$$\beta(Z):=\min\big\{\beta_P(Z)\mid P\in Z\big\}.$$
\end{rmk}

\begin{prop}\label{inversion_of_adjunction}
If $Z$ is a hypersurface in the smooth variety $X$, and $H$ is a smooth hypersurface in $X$ that is not contained in $Z$, 
and if $Z\vert_H=Z\cap H\hookrightarrow H$ is the corresponding hypersurface of $H$, then
$$\beta_P(Z)\geq\beta_P(Z\vert_H)\quad\text{for all}\quad P\in Z\vert_H.$$
\end{prop}

\begin{proof}
We may assume that $Z$ is defined by $f\in\cO_X(X)$, so $g=f\vert_H$ defines $Z_H$. Let $\fm_P$ and $\overline{\fm}_P$ denote
the ideals defining $P$ in $X$ and $H$, respectively.
By choosing a system of algebraic coordinates 
$x_1,\ldots,x_n$ in a neighborhood of $P$ 
such that $H$ is defined by $x_1$, we see that if 
$$\fa=\big(f,\tfrac{\partial f}{\partial x_2},\ldots,\tfrac{\partial f}{\partial x_n}\big),$$
then $\fa\cdot\cO_H=J_g$. Since $\fa\subseteq J_f$, we have
$${\rm lct}_P(\fm_P\fa)\leq {\rm lct}_P(\fm_PJ_f)=\beta_P(Z).$$
On the other hand, since $(\fm_P\fa)\cdot\cO_H=\overline{\fm}_PJ_g$,
it is a consequence of the Restriction Theorem for multiplier ideals (see, for example, \cite[Example~9.5.4]{Lazarsfeld})
that 
$${\rm lct}_P(\fm_P\fa)\geq {\rm lct}_P(\overline{\fm}_PJ_g)=\beta_P(Z_H).$$
By combining the two inequalities, we obtain the assertion in the proposition.
\end{proof}

The next result gives an extension of Proposition~\ref{semicont0} to the case of an arbitrary family of hypersurfaces.

\begin{prop}\label{semicont}
Let $\pi\colon {\mathcal X}\to T$ be a smooth morphism of complex algebraic varieties and suppose that $s\colon T\to {\mathcal X}$ is a morphism
such that $\pi\circ s={\rm Id}_T$. If ${\mathcal Z}\hookrightarrow {\mathcal X}$ is a relative Cartier divisor over $T$, containing $s(T)$, and if for every
(closed point) $t\in T$, we denote by ${\mathcal Z}_t$ the corresponding hypersurface in ${\mathcal X}_t=\pi^{-1}(t)$, then the following hold:
\begin{enumerate}
\item[i)] The set $\big\{\beta_{s(t)}({\mathcal Z}_t)\mid t\in T\big\}$ is finite.
\item[ii)] For every $\beta\in {\mathbf Q}_{\geq 0}$, the set
$$\big\{t\in T\mid \beta_{s(t)}({\mathcal Z}_t)\geq\beta\big\}$$
is open in $T$. 
\end{enumerate}
\end{prop}

\begin{proof}
After taking a resolution of singularities $\widetilde{T}\to T$ and considering instead the morphism ${\mathcal X}\times_T\widetilde{T}\to \widetilde{T}$
and ${\mathcal Z}\times_T\widetilde{T}$, we see that we may assume that $T$ (and thus also ${\mathcal X}$) is smooth. Working locally on ${\mathcal X}$,
we may also assume that ${\mathcal Z}$ is defined by $f$, so ${\mathcal Z}_t$ is defined by $f_t=f\vert_{{\mathcal X}_t}$ for every $t\in T$. 
Furthermore, we may assume 
that we have algebraic coordinates $y_1,\ldots,y_r$ on $T$ and $x_1,\ldots,x_r,\ldots,x_{r+n}$ on ${\mathcal X}$
such that $x_i=y_i\circ\pi$ for $1\leq i\leq r$. Let $f_i=\tfrac{\partial f}{\partial x_{i+r}}$ for $1\leq i\leq n$ and let $\fa=\big(f,f_1,\ldots,f_n)$, so that for every $t\in T$,
we have $J_{f_t}=\fa\cdot\cO_{{\mathcal X}_t}$. If $\fb$ is the ideal defining $s(T)$ in ${\mathcal X}$, then
$$\beta_{s(t)}({\mathcal Z}_t)={\rm lct}_{s(t)}(\fa\fb\cdot\cO_{{\mathcal X}_t}).$$
The assertions in i) and ii) then follow from the behavior of log canonical thresholds of ideals in families (see, for example, \cite[Lemma~4.8 and Theorem~4.9]{Mustata}).
\end{proof}

\begin{prop}\label{m_adic}
If $X$ is a smooth, irreducible, $n$-dimensional algebraic variety, $f$, $g\in\cO_X(X)$ are nonzero, and $P\in X$ is a point such that 
$f(P)=g(P)=0$ and $f-g\in\fm_P^d$, for some positive integer $d$, then 
$$|\beta_P(f)-\beta_P(g)|\leq\tfrac{n}{d}.$$
\end{prop}

\begin{proof}
The hypothesis implies that $J_f+\fm_P^{d-1}=J_g+\fm_P^{d-1}$, hence
$\fm_PJ_f+\fm_P^d=\fm_PJ_g+\fm_P^d$. The assertion in the proposition then follows from the corresponding
property of log canonical thresholds of ideals, see for example \cite[Property~1.21]{Mustata2}.
\end{proof}

\section{Two questions regarding minimal exponents}\label{section3}

As in the previous section, suppose that $X$ is a smooth, irreducible, complex algebraic variety of dimension $n$. We consider a nonzero
$f\in\cO_X(X)$, defining a hypersurface $Z$, and $P\in Z$. In this case, one can associate to $f$ the \emph{Bernstein-Sato polynomial} $b_f(s)$,
which is the monic generator of the ideal consisting of those polynomials $b(s)\in {\mathbf C}[s]$ such that
$$b(s)f^s\in\cD_X[s]\cdot f^{s+1},$$
where $\cD_X$ is the ring of differential operators on $X$. Since $f$ is not invertible, we have $b_f(-1)=0$, and Saito defined the \emph{minimal exponent}
$\widetilde{\alpha}(Z)=\widetilde{\alpha}(f)$ in \cite{Saito-B} as the negative of the largest root of $b_f(s)/(s+1)$ (with the convention that this is $\infty$  if 
$b_f(s)=s+1$, which is known to be the case if and only if $Z$ is a smooth hypersurface). The local version $\widetilde{\alpha}_P(Z)$ is defined
by replacing $X$ by a small open neighborhood of $P$. By a result of Kashiwara \cite{Kashiwara}, it
 is known that, when finite, $\widetilde{\alpha}(Z)$ is a positive rational number (like all roots of $b_f(s)$). Moreover, we have
 $\min\big\{\widetilde{\alpha}_P(Z),1\big\}={\rm lct}_P(Z)$ by a result due to Lichtin and Koll\'{a}r (see \cite[Theorem~10.6]{Kollar}).
 Furthermore, it was shown by Saito in \cite{Saito-B} that $\widetilde{\alpha}(Z)>1$ if and only if $Z$ has rational singularities.
 
 If $Z$ is an arbitrary hypersurface in $X$ and $P\in Z$, then we define $\widetilde{\alpha}_P(Z)$ by choosing a local equation for $Z$ near $P$.
 In general, the minimal exponent satisfies similar properties to the log canonical thresholds. For statements and proofs of some of these results, we refer to \cite{MPV}.

As explained in the Introduction, our main motivation for considering the invariant $\beta_P(Z)$ comes from Question~\ref{question_Kim}, posed
by Dano Kim in \cite{DanoKim}. We begin with a few examples and some remarks in connection with this question.

\begin{eg}
If $Z$ has an ordinary singularity at $P$, of multiplicity $d\geq 2$, then $\beta_P(Z)=\tfrac{n}{d}=\widetilde{\alpha}_P(Z)$,
where the first equality is given in Example~\ref{eg_ordinary} and the second one appears, for example, in \cite[Theorem~E(3)]{MPV}.
Therefore, in this case, we have equality in (\ref{eq_Dano_Kim}).
\end{eg}

\begin{eg}
Suppose that $Z$ is the hypersurface in ${\mathbf A}^n$ defined by $f=x_1^{a_1}+\ldots+x_n^{a_n}$, where
$a_1\geq\ldots\geq a_n\geq 2$. In this case, it follows from the formula for the Bernstein-Sato polynomial of a weighted homogeneous polynomial
with isolated singularities (see \cite{BGM}) that $\widetilde{\alpha}_0(Z)=\sum_{i=1}^n\tfrac{1}{a_i}$. 
On the other hand, it follows from the definition that $\beta_0(Z)={\rm lct}_0(\fa)$, where
$$\fa=(x_1,\ldots,x_n)\cdot (x_1^{a_1-1},\ldots,x_n^{a_n-1}).$$
Since a monomial ideal admits a log resolution given by a toric morphism, in order to compute the log canonical threshold of a monomial ideal,
it is enough to consider monomial valuations over ${\mathbf A}^n$, hence a well-known toric computation gives
$$\beta_0(Z)=\min_{u_1,\ldots,u_n>0}\frac{u_1+\ldots+u_n}{\min_i\{u_i\}+\min_i\{(a_i-1)u_i\}}.$$
In order to check the inequality in (\ref{eq_Dano_Kim}), we thus need to show that for every $p$, $q>0$, and every $u_1,\ldots,u_n$ such that
$u_i\geq p$ and $(a_i-1)u_i\geq q$ for all $i$, we have
\begin{equation}\label{eq_eg_diagonal}
\sum_{i=1}^nu_i\geq (p+q)\cdot\sum_{i=1}^n\tfrac{1}{a_i}.
\end{equation}
However, for every $i$, we have 
$$a_iu_i=u_i+(a_i-1)u_i\geq p+q.$$
Dividing by $a_i$ and summing over all $i$, gives (\ref{eq_eg_diagonal}).
\end{eg}

\begin{eg}
If $Z$ is defined by the determinant of a generic $n\times n$ matrix as in Example~\ref{det}, then we have seen that $\beta_0(Z)=4$,
hence the inequality (\ref{eq_Dano_Kim}) holds in this case since $b_f(s)=\prod_{i=1}^n(s+i)$ implies $\widetilde{\alpha}_0(Z)=2$ (see 
\cite[Theorem~4.1]{Lorincz} for the formula for the Bernstein-Sato polynomial).
\end{eg}

\begin{rmk}
It follows from Remark~\ref{rmk_bound_by_integral_closure} (and it was pointed out already in \cite{DanoKim}) that the inequality 
(\ref{eq_Dano_Kim}) holds if $\widetilde{\alpha}_P(Z)={\rm lct}_P(Z)$. Hence in Question~\ref{question_Kim} we may assume that 
$\widetilde{\alpha}_P(Z)>1$, that is, $Z$ has rational singularities at $P$. 
\end{rmk}

\begin{rmk}\label{isolated_singularity}
In Question~\ref{question_Kim}, we may assume that $Z$ has an isolated singularity at $P$. Indeed, after possibly replacing $X$ by a suitable open neighborhood of $P$, we may assume that we have algebraic coordinates $x_1,\ldots,x_n$
 such that $x_i(P)=0$ for all $i$. For all $d\geq 2$, let $f_d=f+\sum_{i=1}^na_{d,i}x_i^d$, where $a_{d,1},\ldots,a_{d,n}\in {\mathbf C}$ are general. Since the zero locus of
 $(x_1^d,\ldots,x_n^d)$ is $\{P\}$, it follows from the Kleiman-Bertini theorem that the hypersurface defined by $f_d$ is smooth in $X\smallsetminus\{P\}$. By Proposition~\ref{m_adic},
 we have
 $$|\beta_P(f)-\beta_P(f_d)|\leq\tfrac{n}{d},$$
and a similar inequality holds for minimal exponents
$$|\widetilde{\alpha}_P(f)-\widetilde{\alpha}_P(f_d)|\leq\tfrac{n}{d},$$
by \cite[Proposition~6.6(2)]{MPV}. We conclude that if the inequality (\ref{eq_Dano_Kim}) holds for all $f_d$, then it also holds for $f$.
\end{rmk}

Since we know that at every $P\in Z$, $f$ lies in the integral closure of $\fm_PJ_f$, a positive answer to the following optimistic question would provide a positive answer
to Question~\ref{question_Kim}.

\begin{question}\label{question3}
If $X$ is a smooth variety, $\fa$ is an ideal on $X$ cosupported at $P\in X$, and $f\in\cO_X(X)$ is such that $f\in\overline{\fm_P\fa}$, then do we have 
$$\widetilde{\alpha}_P(f)\leq {\rm lct}_P(\overline{\frm_P\fa})\big(={\rm lct}_P(\fm_P\fa)\big)?$$
\end{question}

Of course, Question~\ref{question2} is a weaker version of Question~\ref{question3}, in which we replace the integral closure $\overline{\fm_P\fa}$
with $\fm_P\fa$.

\begin{rmk}
In general, for every ideal $J\subseteq\fm_P^2$, we can ask whether for every $f\in J$, we have $\widetilde{\alpha}_P(f)\leq {\rm lct}_P(J)$ (the condition
$J\subseteq\fm_P^2$ is required in order to make sure that $f$ has a singular point at $P$).
It is very easy to give examples of ideals $J$ for which this fails: for example, take $J=(f)$, where $f$ is such that $\widetilde{\alpha}_P(f)>1$.
In order to get examples with cosupport $P$, it is enough to take $J_d=(f)+\fm_P^d$, with $d\gg 0$, and use the fact that 
$\lim_{d\to\infty}{\rm lct}_P(J_d)={\rm lct}_P(f)$ by \cite[Property~1.21]{Mustata2}.
\end{rmk}

The next result provides a weaker inequality in the context of Question~\ref{question2}.

\begin{thm}\label{prop_question2}
Let $X$ be a smooth, irreducible variety and $P\in X$ a point defined by the ideal $\fm_P$. If $f\in\cO_X(X)$ is nonzero and $\fa\subseteq\fm_P$ is a coherent ideal on $X$ such that
$f\in\fm_P\fa$, then
$$\widetilde{\alpha}_P(f)\leq {\rm lct}_P(\fa).$$
\end{thm}

\begin{proof}
We may and will assume that $X$ is affine. Let $x_1,\ldots,x_r$ be generators of $\fm_P$ and let us write $f=\sum_{i=1}^rg_ix_i$, with $g_1\ldots,g_r\in\fa$. 
If $\fb=(g_1,\ldots,g_r)$, then $\fb\subseteq\fa$, hence it is enough to show that $\widetilde{\alpha}_P(f)\leq {\rm lct}_P(\fb)$. 

Let us consider now
$g=\sum_{i=1}^rg_iy_i\in\cO_Y(Y)$, where $Y=X\times {\mathbf A}^r$ and $y_1,\ldots,y_r$ are the coordinate functions on ${\mathbf A}^r$. 
It follows from \cite[Corollary~1.2]{Mustata3} that
\begin{equation}\label{eq_lct_min_exp}
{\rm lct}(\fb)=\widetilde{\alpha}(g).
\end{equation}

Note that for every $\eta$, the set $\big\{(Q,t)\in Y\mid \widetilde{\alpha}_{(Q,t)}(g)\leq\eta\big\}$
is closed in $Y$ (this follows, for example, from the fact that if $U\subseteq V$ are open subsets, then $b_{f\vert_U}(s)$ divides $b_{f\vert_V}(s)$).
First, this implies that 
after possibly replacing $X$ by an open neighborhood of $P$, we may assume that $\widetilde{\alpha}_{(P,0)}(g)\leq \widetilde{\alpha}_{(Q,0)}(g)$
for all $Q\in X$. Second, since $g$ is homogeneous (of degree 1) with respect to the $y$ variables, this implies that
$\widetilde{\alpha}_{(Q,0)}(g)\leq \widetilde{\alpha}_{(Q,t)}(g)$ for all $Q\in X$ and all $t\in {\mathbf A}^r$. By combining these,
we get
\begin{equation}\label{eq_special_pt}
\widetilde{\alpha}_{(P,0)}(g)=\widetilde{\alpha}(g).
\end{equation}

Let us consider the morphism $\varphi \colon X\to Y$ of varieties over $X$ such that $y_i\circ\varphi=x_i$.  In this case, the behavior of the minimal exponent
with respect to intersecting with a hypersurface gives 
\begin{equation}\label{eq_morphism}
\widetilde{\alpha}_P(f)\leq\widetilde{\alpha}_{(P,0)}(g).
\end{equation}
This is well-known, but we include the argument due to the lack of a reference. Note that $g\circ\varphi=f$. We can factor $\varphi$ as $p\circ i$,
where $i\colon X\hookrightarrow X\times Y$ is the graph of $\varphi$ and $p\colon X\times Y\to Y$ is the projection. We have 
$$\widetilde{\alpha}_{\varphi(P)}(g)=\widetilde{\alpha}_{(P,\varphi(P))}(h),$$
where $h=g\circ p$. On the other hand, applying iteratively \cite[Theorem~E(1)]{MPV}, gives
$$\widetilde{\alpha}_{(P,\varphi(P))}(h)\geq\widetilde{\alpha}_P(h\circ i),$$
which gives (\ref{eq_morphism}) since $h\circ i=f$ and $\varphi(P)=(P,0)$. 
By combining (\ref{eq_lct_min_exp}), (\ref{eq_special_pt}), and (\ref{eq_morphism}), we get
$$\widetilde{\alpha}_P(f)\leq {\rm lct}(\fa)\leq {\rm lct}_P(\fa),$$
completing the proof of the theorem.
\end{proof}

\begin{rmk}\label{rmk_general}
Suppose that $X$ is a smooth affine variety and $P\in X$ is a point defined by the ideal $\fm_P$. Let $J\subseteq\fm_P$ be an ideal generated by $f_1,\ldots,f_r$
(which we may and will assume to be linearly independent over ${\mathbf C}$).
For every $\lambda=(\lambda_1,\ldots,\lambda_r)\in {\mathbf C}^r$, let us put
$g_{\lambda}=\sum_{i=1}^r\lambda_if_i$.  It follows from the semicontinuity property of minimal exponents (see \cite[Theorem~E(2)]{MPV}) that for $\lambda$ general,
the minimal exponent $\widetilde{\alpha}_P(g)$ is constant (say, it is equal to $\alpha$), and $\widetilde{\alpha}_P(g_{\lambda})\leq\alpha$ for all $\lambda\neq 0$. 
Moreover, it follows from \cite[Proposition~3.1]{CKM} that
$\alpha\geq {\rm lct}_P(J)$. 

The following are equivalent:
\begin{enumerate}
\item[i)] $\widetilde{\alpha}_P(f)\leq{\rm lct}_P(J)$ for all nonzero $f\in J$.
\item[ii)] $\alpha={\rm lct}_P(J)$.
\end{enumerate}

The implication i)$\Rightarrow$ii) is clear, and the converse follows if we prove that, in general, for every nonzero $h\in J$, we have $\widetilde{\alpha}_P(h)\leq\alpha$.
In order to show this, as in the proof of Theorem~\ref{prop_question2}, we consider $g=\sum_{i=1}^rf_iy_i\in\cO_Y(Y)$,
where $Y=X\times {\mathbf A}^r$. Note first that by the behavior of minimal exponent when intersecting with general hypersurfaces (see 
\cite[Theorem~1.2ii)]{CDMO}), we have
\begin{equation}\label{eq1_last}
\widetilde{\alpha}_{(P,\lambda)}(g)=\widetilde{\alpha}_P(g_{\lambda})=\alpha\quad\text{for general}\quad\lambda\in {\mathbf A}^r.
\end{equation}

We next show that if $h_1,\ldots,h_r\in\cO_X(X)$ are such that $w=\big(h_1(P),\ldots,h_r(P)\big)$ is general, and $h=\sum_{i=1}^rh_if_i$,
then $\widetilde{\alpha}_P(h)\leq \alpha$. Indeed, consider the morphism $\varphi\colon X\to Y$ given by $\varphi=(1_X,h_1,\ldots,h_r)$. Arguing as in the proof of 
Theorem~\ref{prop_question2}, we see that since $h=g\circ\varphi$, we have
$$\widetilde{\alpha}_P(h)\leq \widetilde{\alpha}_{(P,w)}(g)=\alpha,$$
where the equality follows from (\ref{eq1_last}). 

Suppose now that $h_1,\ldots,h_r\in\cO_X(X)$ are arbitrary such that $h=\sum_{i=1}^rh_if_i$ is nonzero. The semicontinuity property of the minimal exponent
(see \cite[Theorem~E(2)]{MPV}) implies that if $h'=\sum_{i=1}^r(h_i+u_i)f_i$, where $u=(u_1,\ldots,u_r)\in {\mathbf A}^r$ is general, then
$$\widetilde{\alpha}_P(h)\leq \widetilde{\alpha}_P(h')\leq\alpha,$$
where the second inequality follows from the assertion proved in the previous paragraph.
\end{rmk}

\begin{eg}\label{example_monomial}
If $X={\mathbf A}^n={\rm Spec}\,{\mathbf C}[x_1,\ldots,x_n]$ and $\fa\subseteq (x_1,\ldots,x_n)^2$ is a monomial ideal, 
then for every $h\in \fa$, we have $\widetilde{\alpha}_0(h)\leq {\rm lct}_0(\fa)$ (in particular, Questions~\ref{question2} and \ref{question3} have a positive answer 
when $\fa$ is a monomial ideal). Indeed, note first that since 
$$\lim_{d\to\infty}\lct_0(\fa+\fm^d)={\rm lct}_0(\fa),$$ where $\fm=(x_1,\ldots,x_n)$ (see, for example,
\cite[Property~1.21]{Mustata2}), it is enough to prove the assertion when the zero-locus of $\fa$ is $\{0\}$. Furthermore, it follows from Remark~\ref{rmk_general}
that it is enough to show that if $h$ is a general linear combination of the monomial generators of $\fa$, then $\widetilde{\alpha}_0(h)={\rm lct}_0(\fa)$.
However, such $h$ is nondegenerate with respect to its Newton polyhedron and, since the zero-locus of $\fa$ is $\{0\}$, the singular locus of the hypersurface 
defined by $h$ consists of the origin, by the Kleiman-Bertini theorem (the fact that the origin is a singular point is due to the fact that $\fa\subseteq\fm^2$). In this case,
it is known that
\begin{equation}\label{eq_nondeg}
\widetilde{\alpha}_0(h)=1/\min\big\{t>0\mid (t,\ldots,t)\in P(\fa)\big\},
\end{equation}
where $P(\fa)$ is the Newton polyhedron of the ideal $\fa$
(see \cite{Varchenko}, \cite{EhlersLo}, or \cite{Saito-exponents}). The fact that the right-hand side of (\ref{eq_nondeg}) is equal to ${\rm lct}_0(\fa)$
is a consequence of Howald's formula (see \cite[Theorem~9.3.27]{Lazarsfeld}). 
\end{eg}

\end{document}